\documentclass[amstex,12pt,russian,amssymb]{article}

\usepackage{mathtext}
\usepackage[cp1251]{inputenc}
\usepackage[T2A]{fontenc}
\usepackage[russian]{babel}
\usepackage[dvips]{graphicx}
\usepackage{amsmath}
\usepackage{amssymb}
\usepackage{amsxtra}
\usepackage{latexsym}
\usepackage{ifthen}

\begin{document}

\newcounter{lemma}
\newcommand{\lemma}{\par \refstepcounter{lemma}%
{\bf Лемма \arabic{lemma}.}}

\newcounter{corollary}
\newcommand{\corollary}{\par \refstepcounter{corollary}%
{\bf Следствие \arabic{corollary}.}}

\newcounter{remark}
\newcommand{\remark}{\par \refstepcounter{remark}%
{\bf Замечание \arabic{remark}.}}

\newcounter{theorem}
\newcommand{\theorem}{\par \refstepcounter{theorem}%
{\bf Теорема \arabic{theorem}.}}

\newcounter{proposition}
\newcommand{\proposition}{\par \refstepcounter{proposition}%
{\bf Предложение \arabic{proposition}.}}

\renewcommand{\refname}{\centerline{\bf Список литературы}}

\newcommand{\proof}{{\it Доказательство.\,\,}}

\noindent УДК 517.5

{\bf Е.А.~Севостьянов} (Житомирский государственный университет
имени Ивана Франко)

{\bf Є.О.~Севостьянов} (Житомирський державний університет імені
Івана Фран\-ка)

{\bf E.A.~Sevost'yanov} (Zhytomyr Ivan Franko State University)

\medskip
{\bf О неравенстве типа Полецкого для отображений римановых
поверхностей}

{\bf Про нерівність типу Полецького для відображень ріманових
поверхонь}

{\bf On Poletsky type inequality for mappings of Riemannian
surfaces}

\medskip\medskip
В статье получены верхние оценки искажения модуля семейств кривых
при отображениях класса Соболева, дилатация которых локально
интегрируема. Как следствие, получены теоремы о локальном и
граничном поведении указанных отображений.

\medskip\medskip
У статті отримано верхні оцінки спотворення модуля сімей кривих при
відображеннях класу Соболєва, внутрішня дилатація котрих є локально
інтегровною. Як наслідок, отримано теореми про локальну і межову
поведінку вказаних відображень.

\medskip\medskip
In this paper, we obtain upper estimates for the distortion of the
modulus of families of paths under mappings of the Sobolev class,
whose dilatation is locally integrable. As a consequence, theorems
on the local and boundary behavior of the indicated mappings are
obtained.

\newpage
{\bf 1. Введение.} Настоящая статья посвящена изучению отображений с
ограниченным и конечным искажением, активно изучаемых в последнее
время, см. напр.,~\cite{MRV$_1$}, \cite{MRSY} и~\cite{Va}. В
частности, речь идёт об изучении этих отображений в том случае,
когда областями их определения и значения являются некоторые области
римановых поверхностей гиперболического типа. Отметим несколько
последних работ, посвящённых тем же вопросам, см. напр.,~\cite{RV}
и~\cite{RV$_1$}.

Принципиальное значение для исследования отображений имеют оценки
искажения модуля при них (см., напр., \cite[разд.~2.3]{MRV$_1$},
\cite[разд.~4.1]{MRSY}, \cite[определение~13.1]{Va} и
\cite[теорема~3.1]{Va$_2$}). В первую очередь, именно эти оценки
помогают исследовать локальное и граничное поведение отображений
(см., напр.,~\cite[теоремы~17.13, 17.15]{Va},
\cite[теорема~4.2]{Va$_2$}, \cite[теоремы~3.6--3.7]{Na$_2$}). Одна
из них, в частности, получена автором в~\cite{Sev$_2$}. В настоящей
заметке мы продолжаем исследования в этом направлении, получая ещё
более сильное по сравнению с~\cite{Sev$_2$} неравенство типа
Полецкого (см.~\cite[теорема~1]{Pol}).

\medskip
Напомним определения. {\it Римановой поверхностью} будет называться
двумерное многообразие со счётной базой, в котором отображения
перехода между соответствующими картами являются конформными
отображениями, см., напр.,~\cite{RV}. Рассматриваемые ниже римановы
поверхности ${\Bbb S}$ и ${\Bbb S}_*$ будут предполагаться
поверхностями {\it гиперболического типа}, т.е., поверхностями,
конформно-эквивалентными единичному кругу ${\Bbb D}=\{z\in {\Bbb C}:
|z|<1\}$ (см. \cite[\S\,6, разд.~1]{KAG}). Другими словами, мы
рассматриваем те и только те римановы поверхности, которые являются
конформно эквивалентными фактор пространству ${\Bbb D}/G,$ где $G$
-- некоторая группа дробно-линейных автоморфизмов единичного круга
на себя, не имеющая неподвижных точек и действующая разрывно в
${\Bbb D}.$ Напомним, что каждый элемент $p_0$ фактор-пространства
${\Bbb D}/G$ является {\it орбитой} точки $z_0\in {\Bbb D},$ т.е.,
$p=\{z\in {\Bbb D}: z=g(z_0), g\in G\}.$ Всюду далее в единичном
круге ${\Bbb D}$ используется так называемая {\it гиперболическая
метрика:}
\begin{equation}\label{eq3}
h(z_1, z_2)=\log\,\frac{1+t}{1-t}\,,\quad
t=\frac{|z_1-z_2|}{|1-z_1\overline{z_2}|}\,,
\end{equation}
а также
{\it гиперболические площадь} множества $S\subset {\Bbb D}$  и
{длина} кривой $\gamma:[a, b]\rightarrow {\Bbb D},$ которые
задаются, соответственно, соотношениями
\begin{equation}\label{eq1}
h(S)=\int\limits_S\frac{4\,dm(z)}{(1-|z|^2)^2}\,,\quad
z=x+iy\,,\quad
s_h(\gamma):=\sup\limits_{\pi}\sum\limits_{k=0}^nh(\gamma(t_k),
\gamma(t_{k+1}))\,,
\end{equation}
где $h$ из~(\ref{eq3}), а $\sup$ берётся по всем разбиениям
$\pi=\{a=t_0\leqslant t_1\leqslant t_2\leqslant\ldots\leqslant
t_n=b\},$
(см. \cite[(2.4), (2.5)]{RV}). Прямыми вычислениями нетрудно
убедиться, что гиперболические метрика, длина и площадь инвариантны
относительно дробно-линейных отображений единичного круга на себя.

\medskip
Для точки $y_0\in {\Bbb D}$ и числа $r\geqslant 0$ определим {\it
гиперболический круг} $B_h(y_0, r)$ и {\it гиперболическую
окружность} $S_h(y_0, r)$ посредством равенств
\begin{equation}\label{eq7}
B_h(y_0, r):=\{y\in {\Bbb D}: h(y_0, y)<r\}\,, S_h(y_0, r):=\{y\in
{\Bbb D}: h(y_0, y)=r\}\,. \end{equation}
Римановы поверхности можно метризовать следующим образом. Если $p_1,
p_2\in {\Bbb D}/G,$ полагаем
\begin{equation}\label{eq2}
\widetilde{h}(p_1, p_2):=\inf\limits_{g_1, g_2\in G}h(g_1(z_1),
g_2(z_2))\,,
\end{equation}
где $p_i=G_{z_i}=\{\xi\in {\Bbb D}:\,\exists\, g\in G:
\xi=g(z_i)\},$ $i=1,2.$ В последнем случае множество $G_{z_i}$ будем
называть {\it орбитой} точки $z_i,$ а $p_1$ и $p_2$ назовём {\it
орбитами} точек $z_1$ и $z_2,$ соответственно. В дальнейшем
$$\widetilde{B}(p_0, r):=\{p\in {\Bbb S}: \widetilde{h}(p_0, p)<r\}\,,\quad \widetilde{S}(p_0, r):=\{p\in
{\Bbb S}: \widetilde{h}(p_0, p)=r\}$$
-- круг и окружность с центром в точке $p_0$ на поверхности ${\Bbb
S}.$ Всюду далее $B(z_0, r)$ и $S(z_0, r)$ обозначают круг и
окружность с центром в точек $z_0\in {\Bbb C}$ на плоскости.

\medskip
С целью упрощения исследований, введём в рассмотрение так называемое
{\it фундаментальное множество} $F.$ Определим его как подмножество
${\Bbb D},$ содержащее одну и только одну точку орбиты $z\in
G_{z_0}$ (см. \cite[\S\,9.1, гл.~9]{Berd}). {\it Фундаментальной
областью} $D_0$ называется область в ${\Bbb D,}$ обладающая
свойством $D_0\subset F\subset \overline{D_0}$ такая, что
$h(\partial D_0)=0$ (см. там же). Важнейшим примером фундаментальной
области является {\it многоугольник Дирихле},
\begin{equation}\label{eq4}
D_{\zeta}=\bigcap\limits_{g\in G, g\ne I}H_g(\zeta)\,,
\end{equation}
где $H_g(\zeta)=\{z\in {\Bbb D}: h(z, \zeta)<h(z, g(\zeta))\}$ (см.
\cite[соотношение~(2.6)]{RV}). Пусть $\pi$ -- естественная проекция
${\Bbb D}$ на ${\Bbb D}/G,$ тогда $\pi$ -- аналитическая функция,
конформная на $D_0$ (см. также \cite[предложение~9.2.2]{Berd} и
коментарии после~(2.11) в~\cite{RV}). Заметим, кроме того, что между
точками $F$ и ${\Bbb D}/G,$ а значит, и между точками $F$ и ${\Bbb
S},$ существует взаимно однозначное соответствие. В частности, для
измеримого множества $E\subset {\Bbb D}/G$ полагаем
\begin{equation}\label{eq2B}
\widetilde{h}(E):=h(\pi^{\,-1}(E))\,,
\end{equation}
где $h$ -- гиперболическая мера в единичном круге с элементом
площади $dh(z)=\frac{4\,dm(z)}{(1-|z|^2)^2},$ $m$ -- плоская мера
Лебега. Здесь и далее множество $E\subset {\Bbb D}/G$ (или, более
общо, $E\subset {\Bbb S}$) будет называться измеримым, если $E$
можно покрыть счётным числом открытых множеств $U_k,$ $k=1,2,\ldots
,$ гомеоморфных единичному кругу посредством отображения
$\varphi_k:U_k\rightarrow D$ так, что $\varphi_k(U_k\cap E)$
измеримо относительно плоской меры Лебега. Аналогично можно дать
определение борелевского множества $E\subset {\Bbb D}/G$ ($E\subset
{\Bbb S}$). Из определения римановой поверхности вытекает, что
множество $E$ является измеримым (борелевым) тогда и только тогда,
когда множество $\pi^{\,-1}(E)$ измеримо (борелево) в единичном
круге.

\medskip
Пусть $D,$ $D_{\,*}$ -- области на римановых поверхностях ${\Bbb S}$
и ${\Bbb S}_{\,*},$ соответственно.  Обозначим через $\widetilde{h}$
метрику на римановой поверхности ${\Bbb S},$ а через
$\widetilde{h_*}$ -- на римановой поверхности ${\Bbb S}_*.$ Всюду,
если не оговорено противное, мы считаем, что ${\Bbb S}=D/G$ и ${\Bbb
S}_*=D/G_*,$ где $G$ и $G_*$ -- некоторые группы дробно-линейных
автоморфизмов единичного круга. Элементы длины и объёма обозначаются
на поверхностях ${\Bbb S}$ и ${\Bbb S}_*,$ соответственно,
$ds_{\widetilde{h}},$ $d\widetilde{h}$ и $ds_{\widetilde{h_*}},$
$d\widetilde{h_*}.$ Отображение $f:D\rightarrow D_{\,*}$ будет
называться {\it дискретным}, если прообраз $f^{-1}\left(y\right)$
каждой точки $y\in D_{\,*}$ состоит только из изолированных точек.
Отображение $f:D\rightarrow D_{\,*}$ будет называться {\it
открытым}, если образ любого открытого множества $U\subset D$
является открытым множеством в $D_{\,*}.$ Определение отображений
класса Соболева $W_{\rm loc}^{1,1}$ на римановой поверхности может
быть найдено, напр., в работе~\cite{RV}. В дальнейшем для
отображений $f:D\rightarrow D_{\,*}$ класса $W_{\rm loc}^{1,1}$ в
локальных координатах $f_{\overline{z}} = \left(f_x + if_y\right)/2$
и $f_z = \left(f_x - if_y\right)/2,$ $z=x+iy.$ Кроме того, {\it
норма} и {\it якобиан} отображения $f$ в локальных координатах
выражаются, соответственно, как $\Vert
f^{\,\prime}(z)\Vert=|f_z|+|f_{\overline{z}}|$ и
$J_f(z)=|f_z|^2-|f_{\overline{z}}|^2.$ Будем говорить, что $f\in
W_{\rm loc}^{1,2}(D),$ если $f\in W_{\rm loc}^{1,2}$ и, кроме того,
в локальных координатах $\Vert f^{\,\prime}(z) \Vert\in L^2_{\rm
loc}(D).$ {\it Дилатация порядка $p$} отображения $f$ в точке $z$
определяется соотношением
\begin{equation}\label{eq16}
K_f(z)=\frac{|f_z|+|f_{\overline{z}}|}{|f_z|-|f_{\overline{z}}|}
\end{equation}
при $J_f(z)\ne 0,$ $K_f(z)=1$ при $\Vert f^{\,\prime}(z)\Vert=0$ и
$K_f(z)=\infty$ в остальных случаях. Путём непосредственных
вычислений нетрудно убедиться, что величина $K_f(z)$ не зависит от
локальных координат. Как обычно, кривая $\gamma$ на римановой
поверхности ${\Bbb S}$ определяется как непрерывное отображение
$\gamma:I\rightarrow {\Bbb S},$ где $I$ -- конечный отрезок,
интервал либо полуинтервал числовой прямой. Пусть $\Gamma$ --
семейство кривых в ${\Bbb S}.$ Борелевская функция $\rho:{\Bbb
S}\rightarrow [0, \infty]$ будет называться {\it допустимой} для
семейства $\Gamma$ кривых $\gamma,$ если
$\int\limits_{\gamma}\rho(p)\,ds_{\widetilde{h}}(p)\geqslant 1$ для
всякой кривой $\gamma\in \Gamma.$ Последнее коротко записывают в
виде: $\rho\in {\rm adm}\,\Gamma.$ {\it Модулем} семейства $\Gamma$
называется вещественнозначная функция
$$M(\Gamma):=\inf\limits_{\rho\in {\rm adm}\,\Gamma}\int\limits_{\Bbb
S}\rho^2(p)\,d\widetilde{h}(p)\,.$$

Пусть $\Delta \subset \Bbb R$~--- открытый интервал числовой прямой,
$\gamma:  \Delta\rightarrow {\Bbb S}$~--- локально спрямляемая
кривая. В таком случае, очевидно, существует единственная
неубывающая функция длины $l_{\gamma}: \Delta\rightarrow
\Delta_{\gamma}\subset \Bbb{R}$ с условием $l_{\gamma}(t_0)=0,$ $t_0
\in \Delta,$ такая, что значение $l_{\gamma}(t)$ равно длине
подкривой $\gamma\mid_{[t_0, t]}$ кривой $\gamma,$ если $t>t_0,$ и
длине подкривой $\gamma\mid_ {[t,\,t_0]}$ со знаком минус, если
$t<t_0,$ $t\in \Delta.$ Пусть $g:|\gamma|\rightarrow {\Bbb S}_*$~---
непрерывное отображение, где $|\gamma| = \gamma(\Delta)\subset {\Bbb
S}.$ Предположим, что кривая $\widetilde{\gamma}=g\circ \gamma$
также локально спрямляема. Тогда, очевидно, существует единственная
неубывающая функция $L_{\gamma,\,g}: \,\Delta_{\gamma} \rightarrow
\Delta_{\widetilde{\gamma}}$, такая, что
$L_{\gamma,\,g}(l_{\gamma}(t))\,=\,l_{\widetilde{\gamma}}(t)$ при
всех $t\in\Delta.$ Если кривая $\gamma$ задана на отрезке $[a, b]$
или полуинтервале $[a, b),$ то мы будем считать, что $a=t_0.$ Кривая
$\gamma$ называется ({\it полным}) {\it поднятием кривой
$\widetilde{\gamma}$ при отображении $f:D\rightarrow {\Bbb S}_*,$}
если $\widetilde{\gamma}=f \circ \gamma.$

\medskip
Следующее определение может быть найдено
в~\cite[определение~5.2]{Va} либо \cite[разд.~8.4]{MRSY}. Говорят,
что отображение $f: D\rightarrow {\Bbb R}^n$ принадлежит классу
$ACP$ в области $D$ ({\it абсолютно непрерывно на почти всех кривых}
в области $D$), пишем $f\in ACP,$ если для почти всех кривых
$\gamma$ в области $D$ кривая $\widetilde{\gamma}=f\circ\gamma$
локально спрямляема, при этом, функция длины $L_{\gamma,\,f},$
введённая выше, абсолютно непрерывна на всех отрезках, лежащих в
$\Delta_{\gamma}.$ Здесь и далее некое свойство $P$ выполнено для
почти всех кривых, если модуль семейства кривых, для которого это
свойство нарушается, равен нулю.

\medskip
Предположим, что $f: D\rightarrow {\Bbb S}_*$ таково, что никакая
кривая $\alpha\subset D$ при отображении $f$ не переходит в точку.
Тогда (корректно) может быть определена функция
$L^{\,-1}_{\gamma,\,f}.$ В таком случае, будем говорить, что $f$
обладает {\it свойством $ACP^{\,-1}$} в области $D\subset {\Bbb S},$
пишем $f\in ACP^{\,-1},$ если для почти всех кривых
$\widetilde{\gamma}\in f(D)$ каждое поднятие $\gamma$ кривой
$\widetilde{\gamma}$ при отображении $f,$
$f\circ\gamma=\widetilde{\gamma},$ является локально спрямляемой
кривой, и, кроме того, обратная функция $L^{-1}_{\gamma,\,f}$
абсолютно непрерывна на всех отрезках, лежащих в
$\Delta_{\widetilde{\gamma}}$ для почти всех кривых
$\widetilde{\gamma}$ в $f(D)$ и каждого поднятия $\gamma$ кривой
$\widetilde{\gamma}=f\circ\gamma.$ Заметим, что если $f$ --
гомеоморфизм такой, что $f^{\,-1}\in W_{\rm loc}^{1, 2}(f(D)),$
всегда принадлежит классу~$ACP^{\,-1},$ см.~\cite[теорема~28.2]{Va}.
Будем говорить, что отображение $f$ имеет {\it $N$-свойство Лузина},
если $\widetilde{h_*}(f(E))=0$ для любого $E\subset D$ такого, что
$\widetilde{h}(E)=0.$ Аналогично, будем говорить, что отображение
$f$ имеет {\it $N^{\,-1}$-свойство Лузина}, если
$\widetilde{h}(f^{\,-1}(E_*))=0$ для любого $E_*\subset D_*$ такого,
что $\widetilde{h_*}(E_*)=0.$ Имеет место следующее утверждение, см.
также~\cite[лемма~3.1]{RV}.

\medskip
\begin{theorem}\label{th1}{\sl\, Пусть $D$
и $D_{\,*}$~--- области римановых поверхностей ${\Bbb S}$ и ${\Bbb
S}_*,$ соответственно, при этом, $\overline{D}$ и
$\overline{D_{\,*}}$ являются компактами. Пусть также $f$ --
дифференцируемое почти всюду отображение области $D$ на $D_*,$
принадлежащее классу~$ACP^{\,-1}$ и обладающее $N$ и
$N^{\,-1}$-свойствами Лузина. Тогда для каждого семейства (локально
спрямляемых) кривых $\Gamma$ в области $D$ и каждой допустимой
функции $\rho\in{\rm adm}\,\Gamma$ выполнено неравенство
\begin{equation}\label{eq1A}
M(f(\Gamma))\leqslant
\int\limits_{D}K_f(p)\cdot\rho^2(p)\,d\widetilde{h}(p)\,.
\end{equation}
}
\end{theorem}

\medskip
В силу~\cite[теорема~28.2]{Va} и \cite[следствие~B]{MM}, имеем также
следующее

\medskip
\begin{corollary}\label{cor1}
{\sl\, Пусть $D$ и $D_{\,*}$~--- области римановых поверхностей
${\Bbb S}$ и ${\Bbb S}_*,$ соответственно, при этом, $\overline{D}$
и $\overline{D_{\,*}}$ являются компактами. Пусть также $f$ --
области $D$ на $D_*,$ такие что $f\in W_{\rm loc}^{1, 2}(D)$ и
$f^{\,-1}\in W_{\rm loc}^{1, 2}(f(D)).$ Тогда выполняется
соотношение~(\ref{eq1A}).}
\end{corollary}

\medskip
{\bf 2. Предварительные замечания.} Прежде, чем переходить к
вспомогательным утверждениям и доказательству основных результатов,
сделаем некоторые важные замечания. Предположим, $F$ и $D_0$ --
некоторые фундаментальные множество и область, соответственно (см.
замечания, сделанные во введении). Пусть $\pi$ -- естественная
проекция ${\Bbb D}$ на ${\Bbb D}/G,$ тогда $\pi$ -- аналитическая
функция, конформная на $D_0$ (см. также
\cite[предложение~9.2.2]{Berd} и коментарии после~(2.11)
в~\cite{RV}). Для $z_1, z_2\in F$ положим
\begin{equation}\label{eq5}
d(z_1, z_2):=\widetilde{h}(\pi(z_1), \pi(z_2))\,,
\end{equation}
где $\widetilde{h}$ определено в~(\ref{eq2}). Заметим, что по
определению $d(z_1, z_2)\leqslant h(z_1, z_2).$ Покажем, что для
любого компакта $A\subset {\Bbb D}$ найдётся $\delta=\delta(A)>0:$
\begin{equation}\label{eq34}
d(z_1, z_2)=h(z_1, z_2), \quad \forall\,\, z_1, z_2\in A: h(z_1,
z_2)<\delta\,.
\end{equation}
Предположим противное. Тогда для произвольного $k\in {\Bbb N}$
найдутся комплексные числа $x_k, z_k\in A$ такие, что $h(z_k,
x_k)<1/k$ и, при этом, $d(z_k, x_k)<h(z_k, x_k).$ Тогда по
определению метрики $d$ и инвариантности метрики $h$ при
дробно-линейных отображениях единичного круга на себя, найдётся
$g_k\in G$ такое, что
\begin{equation}\label{eq37}
d(z_k, x_k)\leqslant h(z_k, g_k(x_k))<h(z_k, x_k)<1/k,\quad g_k\in
G, \quad k=1,2,\ldots \,.
\end{equation}
Поскольку $A$ -- компакт в ${\Bbb D},$ мы можем считать, что $x_k,
z_k\rightarrow x_0\in {\Bbb D}$ при $k\rightarrow\infty.$ Тогда из
(\ref{eq37}) по неравенству треугольника имеем $h(g_k(x_k),
x_0)\leqslant h(g_k(x_k), z_k)+h(z_k, x_k)\rightarrow 0$ при
$k\rightarrow\infty,$ и, значит, $h(x_k, g_k^{\,-1}(x_0))\rightarrow
0$ при $k\rightarrow\infty,$ поскольку метрика $h$ инвариантна при
дробно-линейном отображении. Но тогда также по неравенству
треугольника $h(g_k^{\,-1}(x_0), x_0)\leqslant h(g_k^{\,-1}(x_0),
x_k)+h(x_k, x_0)\rightarrow 0,$ $k\rightarrow\infty.$ Последнее
противоречит разрывности группы $G$ в ${\Bbb D},$ что и доказывает
(\ref{eq34}).

\medskip
Имеет место следующая

\medskip
\begin{lemma}\label{lem3}
{\sl Пусть $0<2r_0<1,$ тогда найдётся постоянная $C_1=C_1(r_0)$
такая, что
\begin{equation}\label{eq9B}
C_1\cdot h(z_1, z_2)\leqslant |z_1-z_2|\leqslant  h(z_1,
z_2)\quad\forall\,\, z_1, z_2\in B(0, r_0)\,. \end{equation}
Более того, правое неравенство в~(\ref{eq9B}) имеет место при всех
$z_1, z_2\in {\Bbb D}.$}
\end{lemma}

\medskip
\begin{proof}
Заметим, что по неравенству треугольника $0<|z_1-z_2|<2r_0.$ Поэтому
$r:=|z_1-z_2|$ изменяется в пределах от $0$ до $2r_0<1.$ Напомним,
что $h(z_1,
z_2)=\log\frac{1+\frac{|z_1-z_2|}{|1-z_1\overline{z_2}|}}{1-\frac{|z_1-z_2|}{|1-z_1\overline{z_2}|}}.$
Обозначая $r=|z_1-z_2|,$ заметим, что $h(z_1, z_2)\geqslant
\log\frac{1+r/2}{1-r/2}.$ Заметим, что
\begin{equation}\label{eq8C}
h(z_1, z_2)\geqslant \log\frac{1+r/2}{1-r/2}\geqslant r\,,\quad r\in
(0, 1)\,.
\end{equation}
В самом деле, функция $\varphi(r)=\log\frac{1+r/2}{1-r/2}-r$
возрастает по $r\in [0, 1],$ что проверяется взятием производной.
Значит, её минимум достигается при $r=0,$ т.е., $\varphi(r)\geqslant
0$ при всех $r\in (0, 1)$ и, значит, имеет место
неравенство~(\ref{eq8C}).

Установим левое неравенство в~(\ref{eq9B}). Для этого заметим, что
$h(z_1, z_2)\leqslant \log\frac{1-r^2_0+r}{1-r^2_0-r},$
$\log\frac{1-r^2_0+r}{1-r^2_0-r}\sim \frac{2}{1-r^2_0}\cdot r$ при
$r\rightarrow 0.$ Тогда при некотором $0<r_1<r_0$ и некотором
$M=M(r_0)$
$$h(z_1, z_2)\leqslant
\log\frac{1-r^2_0+r}{1-r^2_0+r}\leqslant M r\,,\quad r\in (0,
r_2)\,.$$
При $r\in [0, 1]$ функция $1-r^2_0-r$ строго положительна по $r.$
Поэтому функция $\frac1r\cdot\log\frac{1-r^2_0+r}{1-r^2_0-r}$
непрерывна на $r\in [r_1, 1]$ и, значит, ограничена при тех же $r$ с
некоторой постоянной $\widetilde{C}.$ Полагая $C^{\,-1}_1:=\max\{M,
\widetilde{C}\},$ получаем, что
\begin{equation}\label{eq8A}
h(z_1, z_2)\leqslant \log\frac{1-r^2_0+r}{1-r^2_0+r}\leqslant
C^{\,-1}_1\cdot r=C^{\,-1}_1\cdot |z_1-z_2|\qquad\forall\,\, z_1,
z_2\in B(0, r_0)\,.
\end{equation}
Лемма доказана.~$\Box$
\end{proof}

\medskip
Докажем следующее важное утверждение,
обобщающее~\cite[теорема~1.3(5)]{Va}.

\medskip
\begin{lemma}\label{lem2}
{\sl\,Предположим, кривая $\alpha:[a, b]\rightarrow {\Bbb D}$
спрямляема в смысле гиперболической длины $s_h$ в~(\ref{eq1}), кроме
того, $s_h=s_h(t)$ обозначает гиперболическую длину кривой $\alpha,$
подсчитанную на отрезке $[a, t],$ $a\leqslant t\leqslant b.$ Тогда
$\alpha^{\,\prime}(t)$ и $s_h^{\,\prime}(t)$ существуют при почти
всех $t\in [a, b],$ при этом,
\begin{equation}\label{eq8B}
\frac{2|\alpha^{\,\prime}(t)|}{1-|\alpha(t)|^2}=s_h^{\,\prime}(t)
\end{equation}
при почти всех $t\in [a, b].$
 }
\end{lemma}

\medskip
\begin{proof}
Функция $s_h=s_h(t)$ монотонна и потому почти всюду дифференцируема.
Кроме того, поскольку $\alpha(t)$ спрямялема, то найдётся $0<r_0<1$
такое, что $\alpha(t)\in B(0, r_0)$ при всех $t\in [a, b].$ Тогда из
леммы~\ref{lem3} вытекает, что кривая $\alpha$ также спрямляема в
евклидовом смысле, поэтому имеет ограниченную вариацию и, значит,
также дифференцируема почти всюду.

\medskip
Чтобы установить равенство~(\ref{eq8B}), будем следовать логике
рассуждений, использованных при
доказательстве~\cite[теорема~1.3(5)]{Va}. Прежде всего, исходя их
определения гиперболической длины кривой в~(\ref{eq1}), мы можем
записать, что
\begin{equation}\label{eq9}
\frac{h(\alpha(t), \alpha(t_0))}{|t-t_0|}\leqslant
\frac{|s_h(t)-s_h(t_0)|}{|t-t_0|}\,.
\end{equation}
Домножая числитель и знаменатель соотношения~(\ref{eq9}) на
$|\alpha(t)-\alpha(t_0)|,$ мы получим, что
\begin{equation}\label{eq9A}\frac{|\alpha(t)-\alpha(t_0)|}{|\alpha(t)-\alpha(t_0)|}
\cdot \frac{h(\alpha(t), \alpha(t_0))}{|t-t_0|}\leqslant
\frac{|s_h(t)-s_h(t_0)|}{|t-t_0|}\,.
\end{equation}
Выясним поведение функции $\varphi(t)=\frac{h(\alpha(t),
\alpha(t_0))}{|\alpha(t)-\alpha(t_0)|}$ при $t\rightarrow t_0.$
Поскольку $\log\frac{1+x}{1-x}\sim 2x$ при $x\rightarrow 0,$ то
$$\varphi(t)=\log\left(\frac{1+\frac{|\alpha(t)-\alpha(t_0)|}{|1-\alpha(t)\overline{\alpha(t_0)}|}}
{1-\frac{|\alpha(t)-\alpha(t_0)|}{|1-\alpha(t)\overline{\alpha(t_0)}|}}\right)\cdot\frac{1}{|\alpha(t)-\alpha(t_0)|}\sim
\frac{2|\alpha(t)-\alpha(t_0)|}{|1-\alpha(t)\overline{\alpha(t_0)}|}\cdot
\frac{1}{|\alpha(t)-\alpha(t_0)|}$$
при $t\rightarrow t_0.$ Тогда $\varphi(t)\rightarrow
\frac{2}{1-|\alpha(t_0)|^2}$ при $t\rightarrow t_0.$ В таком случае,
переходя в~(\ref{eq9A}) к пределу при $t\rightarrow t_0,$ получаем,
что
\begin{equation}\label{eq10}
\frac{2|\alpha^{\,\prime}(t)|}{1-|\alpha(t_0)|^2}\leqslant
s_h^{\,\prime}(t)
\end{equation}
при почти всех $t\in [a, b].$

Для завершения доказательства осталось установить противоположное
к~(\ref{eq10}) неравенство. Обозначим через $A$ множество всех точек
отрезка $[a, b],$ для которых $\alpha^{\,\prime}(t)$ и
$s_h^{\,\prime}(t)$ существуют и, при этом,
$$\frac{2|\alpha^{\,\prime}(t)|}{1-|\alpha(t_0)|^2}<
s_h^{\,\prime}(t)$$ Пусть $A_k$~--- множество всех точек $t\in A,$
для которых
$$\frac{s_h(q)-s_h(p)}{q-p}\geqslant \frac{h(\alpha(q),\alpha(p))}{q-p}+1/k\,,$$
где $a\leqslant p\leqslant t\leqslant q \leqslant b$ и $0<q-p<1/k.$
Ясно, что для завершения доказательства достаточно установить, что
$m_1(A_k)=0$ при всяком $k=1,2,\ldots,$ где $m_1$~--- мера Лебега в
${\Bbb R}^1.$

Пусть $l(\alpha)$ означает длину кривой $\alpha.$ Для произвольного
$\varepsilon>0$ рассмотрим разбиение отрезка $[a, b]$ точками
$a=t_1\leqslant t_2\leqslant\ldots\leqslant t_m=b$ так, что
$l(\alpha)\leqslant \sum\limits_{k=1}^m h(\alpha(t_k),
\alpha(t_{k-1}))+\varepsilon/k$ и $t_{j}-t_{j-1}<1/k$ при всех
$j=1,2,\ldots, m.$ Если $[t_{j-1}, t_j]\cap A_k\ne \varnothing,$ то
по определению множества $A_k,$ $s_h(t_j)-s_h(t_{j-1})\geqslant
h(\alpha(t_j), \alpha(t_{j-1}))+(t_{j}-t_{j-1})/k.$ Следовательно,
обозначив $\Delta_j:=[t_{j-1}, t_j],$ будем иметь:
$$m_1(A_k)\leqslant\sum\limits_{\Delta_j\cap A_k\ne \varnothing}m_1(\Delta_j)\leqslant
k\sum\limits_{j=1}^m (s_h(t_j)-s_h(t_{j-1})- h(\alpha(t_j),
\alpha(t_{j-1}))) \leqslant$$
$$\leqslant k\left(l(\alpha)-
\sum\limits_{j=1}^m h(\alpha(t_j), \alpha(t_{j-1}))\right)\leqslant
\varepsilon\,.$$
Последнее соотношение доказывает равенство $m_1(A_k)=0$ а, значит,
поскольку $A=\bigcup\limits_{k=1}^{\infty} A_k,$ $m_1(A)=0,$ что и
требовалось установить.~$\Box$
\end{proof}

\medskip
Пусть $I$ -- открытый, замкнутый или полузамкнутый конечный интервал
числовой прямой. Согласно~\cite[разд.~7.1]{He}
и~\cite[теорема~2.4]{Va}, произвольная спрямляемая кривая
$\gamma:I\rightarrow {\Bbb C}$ (соответственно, $\gamma:I\rightarrow
{\Bbb S}$) допускает параметризацию $\gamma(t)=(\gamma^0\circ
l_\gamma)(t),$ где $l_\gamma$ обозначает длину кривой $\gamma$ на
отрезке $[a, t].$ В зависимости от контекста эта длина может
пониматься как в евклидовом, так и в гиперболическом смысле, а также
в смысле римановой поверхности. В этом случае, кривая $\gamma^0:[0,
l(\gamma)]\rightarrow {\Bbb C}$ (соответственно, $\gamma^0:[0,
l(\gamma)]\rightarrow {\Bbb S}$) единственна и называется {\it
нормальным представлением} кривой~$\gamma.$ На основании сказанного,
из леммы~\ref{lem2} получаем следующее очевидное утверждение.

\medskip
\begin{corollary}\label{cor2}
{\sl\,Пусть $\alpha:[a, b]\rightarrow {\Bbb D}$ -- абсолютно
непрерывная кривая и $\rho:{\Bbb D}\rightarrow {\Bbb R}$ --
неотрицательная борелевская функция. Тогда
\begin{equation}\label{eq14}
\int\limits_{\alpha}\rho(x)\,ds_h(x)=\int\limits_a^b\frac{2\rho(\alpha(t))
|\alpha^{\,\prime}(t)|}{1-|\alpha(t)|^2}\,dt\,, \end{equation}
в частности,
\begin{equation}\label{eq10A}
l(\alpha)=\int\limits_a^b\frac{2
|\alpha^{\,\prime}(t)|}{1-|\alpha(t)|^2}\,dt\,.
\end{equation}
}
\end{corollary}

\medskip
Пусть $p_0\in {\Bbb S}$ и $z_0\in {\Bbb D}$ таково, что
$\pi(z_0)=p_0,$ где $\pi$ -- естественная проекция ${\Bbb D}$ на
${\Bbb D}/G.$ Обозначим через $D_0$ многоугольник Дирихле с центром
в точке $z_0,$ и положим $\varphi:=\pi^{\,-1}.$ Заметим, что
отображение $\varphi$ является гомеоморфизмом $({\Bbb S},
\widetilde{h})$ на $(F, d),$ где $\widetilde{h}$ -- метрика на
поверхности ${\Bbb S},$ а $d$ -- вышеопределённая метрика на
фундаментальном множестве $F,$ $D_0\subset F\subset\overline{D_0}.$
Не ограничивая общности, можно также считать, что $z_0=0.$ В самом
деле, в противном случае рассмотрим вспомогательное отображение
$g_0(z)=(z-z_0)/(1-z\overline{z_0}),$ не имеющее неподвижных точек
внутри единичного круга. Тогда, если $G$ -- группа дробно-линейных
отображений, соответствующая поверхности ${\Bbb S},$ то
$G^{\,\prime}=\{g_0\circ g, g\in G\},$ очевидно, также соответствует
${\Bbb S}$ в том смысле, что поверхность ${\Bbb S}$ снова является
конформно эквивалентной фактор-пространству ${\Bbb D}/G^{\,\prime}.$
Выберем компактную окрестность $V\subset {\Bbb D}$ точки $0\in
F\subset {\Bbb D},$ такую что $d(x, z)=h(x, z)$ при всех $x, z\in
V,$ что возможно ввиду условия~(\ref{eq34}). Кроме того, выберем $V$
так, чтобы $V\subset B(0, r_0)$ при некотором $0<r_0<1.$ Положим
$U:=\pi(V).$ Окрестность $U$ в этом случае назовём {\it нормальной
окрестностью точки $p_0.$}

Учитывая~\cite[разд.~8, леммы~8.2 и 8.3]{MRSY}, переходя к покрытию
римановой поверхности конечным или счётным числом нормальных
окрестностей и используя счётную полуаддитивность меры
$\widetilde{h},$ получаем следующее утверждение.

\medskip
\begin{proposition}\label{pr1}
{\sl\, Пусть отображение $f:D\rightarrow {\Bbb S}_*$ почти всюду
дифференцируемо в локальных координатах и, кроме того, обладает $N$
и $N^{\,-1}$-свой\-ст\-вами Лузина. Тогда найдётся не более чем
счётная последовательность компактных множеств $C_k^{\,*}\subset D,$
такая что $\widetilde{h}(B)=0,$ где $B=D\setminus
\bigcup\limits_{k=1}^{\infty} C_k^{\,*}$ и $f|_{C_k^{\,*}}$ взаимно
однозначно и билипшицево в локальных координатах для каждого
$k=1,2,\ldots .$ Более того, $f$ дифференцируемо при всех $x\in
C_k^{\,*}$ и выполнено условие $J_f(x)\ne 0.$}
\end{proposition}

\medskip
Пусть $\gamma:[a, b]\rightarrow {\Bbb S}$ -- (локально спрямляемая)
кривая на римановой поверхности ${\Bbb S}.$ Тогда определим функцию
$l_{\gamma}(t)$ как длину кривой $\gamma|_{[a, t]},$ $a\leqslant
t\leqslant b$ (где <<длина>> понимается в смысле римановой
поверхности). Для произвольного множества $B\subset {\Bbb S}$
положим
\begin{equation}\label{eq36}
l_{\gamma}(B)={\rm mes}_1\,\{s\in [0, l(\gamma)]: \gamma(s)\in
B\}\,,
\end{equation}
где, как обычно, ${\rm mes}_1$ обозначает линейную меру Лебега в
${\Bbb R},$ а $l(\gamma)$ -- длина~$\gamma.$ Аналогично можно
определить величину $l_{\gamma}(B)$ для штриховой линии $\gamma,$
т.е., когда $\gamma:~\bigcup\limits_{i=1}^{\infty}(a_i,
b_i)\rightarrow {\Bbb S},$ где $a_i<b_i$ при всех $i\in {\Bbb N}$ и
$(a_i, b_i)\cap (a_j, b_j)=\varnothing$ при всех $i\ne j.$

\medskip
Докажем теперь следующее утверждение, см.
также~\cite[теорема~33.1]{Va}.

\medskip
\begin{lemma}\label{lem1}
{\sl\, Предположим, множество $B_0\subset {\Bbb S}$ имеет нулевую
$\widetilde{h}$-меру. Тогда для почти всех кривых $\gamma$ в ${\Bbb
S}$
\begin{equation}\label{eq18}
l_{\gamma}(B_0)=0\,.
\end{equation}
}
\end{lemma}
\begin{proof}
Ввиду регулярности лебеговой меры найдутся борелево множество
$B\subset {\Bbb S}$ такое, что $B_0\subset B$ и
$\widetilde{h}(B_0)=\widetilde{B}=0,$ где $\widetilde{h}$ -- мера на
поверхности ${\Bbb S},$ определённая соотношением~(\ref{eq2B}).
Рассмотрим покрытие поверхности ${\Bbb S}$ всевозможными шарами вида
$\widetilde{B}(x_0, r_0),$ где $r_0=r(x_0)>0$ таково, что
$\widetilde{B}(x_0, r_0)$ лежит в некоторой нормальной окрестности
$U$ точки $x_0.$ Поскольку по предположению ${\Bbb S}$ --
пространство со счётной базой, по~\cite[теорема Линделёфа,
1.5.XI]{Ku$_1$} можно выделить последовательность точек $x_i,$
$i=1,2,\ldots ,,$ и соответствующих им радиусов шаров
$r_i=r_i(x_i),$ $i=1,2,\ldots ,,$ таких, что
$${\Bbb S}=\bigcup\limits_{i=1}^{\infty} \widetilde{B}(x_i, r_i)\,,\quad
\overline{\widetilde{B}}(x_i, r_i)\subset U_i\,.$$
Обозначим через $\varphi_i=\pi_i^{\,-1}$ отображение,
соответствующее определению нормальной окрестности $U_i$ (см.
комментарии, сделанные перед предложением~\ref{pr1}). Пусть $g_i$ --
характеристическая функция множества $\varphi_i(B\cap U_i).$
Согласно~\cite[теорема~3.2.5, $m=1$]{Fe}, будем иметь:
\begin{equation}\label{eq19}
\int\limits_{\varphi_i(\gamma)}g_i(z)|dz|={\mathcal
H}^{\,1}(\varphi_i(B\cap |\gamma|))\,,
\end{equation}
где $\gamma:[a, b]\rightarrow {\Bbb S}$ -- произвольная локально
спрямляемая кривая, $|\gamma|$ -- носитель кривой $\gamma$ в ${\Bbb
S},$ а $|dz|$ -- элемент евклидовой длины. Рассуждая аналогично
доказательству~\cite[теорема~33.1]{Va}, полагаем
$$\rho(p)= \left \{\begin{array}{rr}\infty, & p\in B,
\\ 0 ,  &  p \notin B\ .
\end{array} \right.$$
Заметим, что $\rho$ -- борелева функция. Пусть $\Gamma_i$ --
подсемейство всех кривых из $\Gamma,$ для которых ${\mathcal
H}^{\,1}(\varphi_i(B\cap |\gamma|))>0.$ Ввиду~(\ref{eq19}) для
каждой $\gamma\in\Gamma_i$ имеем:
$$\int\limits_{\gamma\cap \widetilde{B}(x_i, r_i)}\rho(p)\,ds_{\widetilde{h}}(p)=\int\limits_{\varphi_i(\gamma)}
\rho(\pi_i(y))\,ds_h(y)=2\int\limits_{\varphi_i(\gamma)}\frac{\rho(\pi_i(y))}{1-|y|^2}|dy|=$$
$$=2\int\limits_{\varphi_i(\gamma)}\frac{g(y)\rho(\pi_i(y))}{1-|y|^2}|dy|=\infty\,,$$
где $\gamma\cap \widetilde{B}=\gamma|_{S_i}$ -- штриховая линия,
$S_i=\{s\in [0, l(\gamma)]: \gamma(s)\in \widetilde{B}(x_i, r_i)\}.$
Тогда $\rho\in{\rm adm}\,\Gamma_i.$ Таким образом,
\begin{equation}\label{eq5A}
M(\Gamma_i)\leqslant \int\limits_{\Bbb
S}\rho^2(p)\,d\widetilde{h}(p)=0\,.
\end{equation}
Заметим, что
$\Gamma>\bigcup\limits_{i=1}^{\infty}\Gamma_i,$ поэтому из
соотношения~(\ref{eq5A}) вытекает, что
$M(\Gamma)\leqslant\sum\limits_{i=1}^{\infty}M(\Gamma_i)=0.$ Лемма
доказана.~$\Box$
\end{proof}

\medskip
Пусть $f:D\rightarrow {\Bbb C}$ (либо $f:D\rightarrow {\Bbb S}$) --
отображение, для которого образ никакой кривой в $D$ не вырождается
в точку. Пусть $I_0$ -- отрезок и $\beta: I_0\rightarrow {\Bbb C}$
(либо $\beta: I_0\rightarrow {\Bbb S}$) -- спрямляемая кривая. Пусть
также $\alpha: I\rightarrow D$ -- некоторая кривая, такая, что
$f\circ \alpha\subset \beta.$ Если функция длины $l_{\beta}:
I_0\rightarrow [0, l(\beta)]$ постоянна на некотором интервале
$J\subset I,$ то $\beta$ постоянна на $J$ и, ввиду предположения
относительно $f,$ кривая $\alpha$ также постоянна на $J.$ Отсюда
вытекает, что существует единственная кривая $\alpha^{\,*}:
l_\beta(I)\rightarrow D$, такая что $\alpha=\alpha^{\,*}\circ
(l_\beta|_I).$ Будем говорить, что $\alpha^{\,*}$ является {\it
$f$-пред\-став\-лением кривой $\alpha$ относительно~$\beta.$}

\medskip
{\bf 3. Доказательство теоремы~\ref{th1}.} Пусть $B_0$ and $C_k^*,$
$k=1,2,\ldots ,$ -- множества, соответствующие обозначениям
предложения~\ref{pr1}. Полагая $B_1=C_1^*,$ $B_2=C_2^*\setminus
B_1,\ldots ,$
\begin{equation} \label{eq7.3.7y} B_k=C_k^*\setminus
\bigcup\limits_{l=1}\limits^{k-1}B_l\,,
\end{equation}
мы получим счётное покрытие области $D$ попарно непересекающимися
множествами $B_k, k=0,1,2,\ldots $ такое, что
$\widetilde{h}(B_0)=0,$ $B_0=D\setminus
\bigcup\limits_{k=1}^{\infty} B_k.$ Поскольку по предположению
отображение $f$ обладает $N$-свойством в $D,$ то
$\widetilde{h_*}(f(B_0))=0.$

Поскольку $\overline{D}$ и $\overline{D_*}$ -- компакты, то
существуют конечные покрытия $U_i,$ $1\leqslant i\leqslant I_0$ и
$V_n,$ $1\leqslant n\leqslant N_0,$ такие, что
$$\overline{D}\subset\bigcup\limits_{i=1}^{I_0}U_i\,,\quad \overline{D_*}\subset\bigcup
\limits_{n=1}^{N_0}V_n\,,$$
где $U_i$ и $V_n$ -- нормальные окрестности некоторых точек $x_i\in
{\Bbb S}$ и $y_n\in {\Bbb S}_*.$ Можно выбрать эти покрытия таким
образом, чтобы $\widetilde{h}(\partial U_i)=\widetilde{h_*}(\partial
V_n)=0$ при каждых $1\leqslant i\leqslant I_0$ и $V_n,$ $1\leqslant
n\leqslant N_0.$ В частности, найдутся конформные отображения
$\varphi_i: U_i\rightarrow B(0, r_i),$ $0<r_i<1,$ и $\psi_n:
V_n\rightarrow B(0, R_n),$ $0<R_n<1,$ такие что длина и площадь в
$U_i$ и $V_n$ вычисляются при помощи карт $\varphi_i$ и $\psi_n$
согласно формул~(\ref{eq1}) и~(\ref{eq10A}). Положим
\begin{equation}\label{eq13}
R_0:=\max\limits_{1\leqslant n\leqslant N_0}R_n\,,\qquad
r_0:=\max\limits_{1\leqslant i\leqslant I_0}r_i\,.
\end{equation}
Положим теперь $U_1^{\,\prime}=U_1,$ $U^{\,\prime}_2=U_2\setminus
\overline{U_1},$ $U^{\,\prime}_3=U_3\setminus (\overline{U_1}\cup
\overline{U_2}),$ $\ldots, U^{\,\prime}_{I_0}=U_{I_0}\setminus
(\overline{U_1}\cup \overline{U_2}\ldots \overline{U_{I_0-1}}).$
Заметим, что по определению $U^{\,\prime}_i\subset U_i$ при
$1\leqslant i\leqslant I_0$ и $U^{\,\prime}_i\cap
U^{\,\prime}_j=\varnothing$ при $i\ne j.$ Кроме того,
$D=\left(\bigcup\limits_{i=1}^{I_0}U^{\,\prime}_i\right)\bigcup
B_0^{\,*},$ где $U^{\,\prime}_i$ открыты, а
$\widetilde{h}(B_0^{\,*})=0.$

\medskip
Аналогично, положим $V_1^{\,\prime}=V_1,$
$V^{\,\prime}_2=V_2\setminus \overline{V_1},$
$V^{\,\prime}_3=V_3\setminus (\overline{V_1}\cup \overline{V_2}),$
$\ldots, V^{\,\prime}_{N_0}=V_{N_0}\setminus (\overline{V_1}\cup
\overline{V_2}\ldots \overline{V_{N_0-1}}).$ По определению
$V^{\,\prime}_n\subset V_n$ при $1\leqslant n\leqslant N_0$ и
$V^{\,\prime}_n\cap V^{\,\prime}_j=\varnothing$ при $n\ne j.$ Кроме
того,
$D_*=\left(\bigcup\limits_{n=1}^{N_0}V^{\,\prime}_n\right)\bigcup
B_0^{\,**},$ где $V^{\,\prime}_n$ открыты, а
$\widetilde{h_*}(B_0^{\,**})=0.$

\medskip
Далее положим $U_{n, i}=f^{\,-1}(V^{\,\prime}_n)\cap
U^{\,\prime}_i.$ Заметим, что по построению и непрерывности $f$
множества $U_{n, i}$ являются открытыми. Кроме того, по
$N^{\,-1}$-свойству $\widetilde{h}(f^{\,-1}(B_0^{\,**}))=0.$ Таким
образом,
\begin{equation}\label{eq7A}
D\subset \left(\bigcup\limits_{1\leqslant i\leqslant I_0\atop
1\leqslant n\leqslant N_0}U_{n, i}\right)\bigcup
f^{\,-1}(B_0^{\,**})\bigcup B_0^{\,*}\,,
\end{equation}
см. рисунок~\ref{fig1} для иллюстрации.
\begin{figure}[h]
\centerline{\includegraphics[scale=0.5]{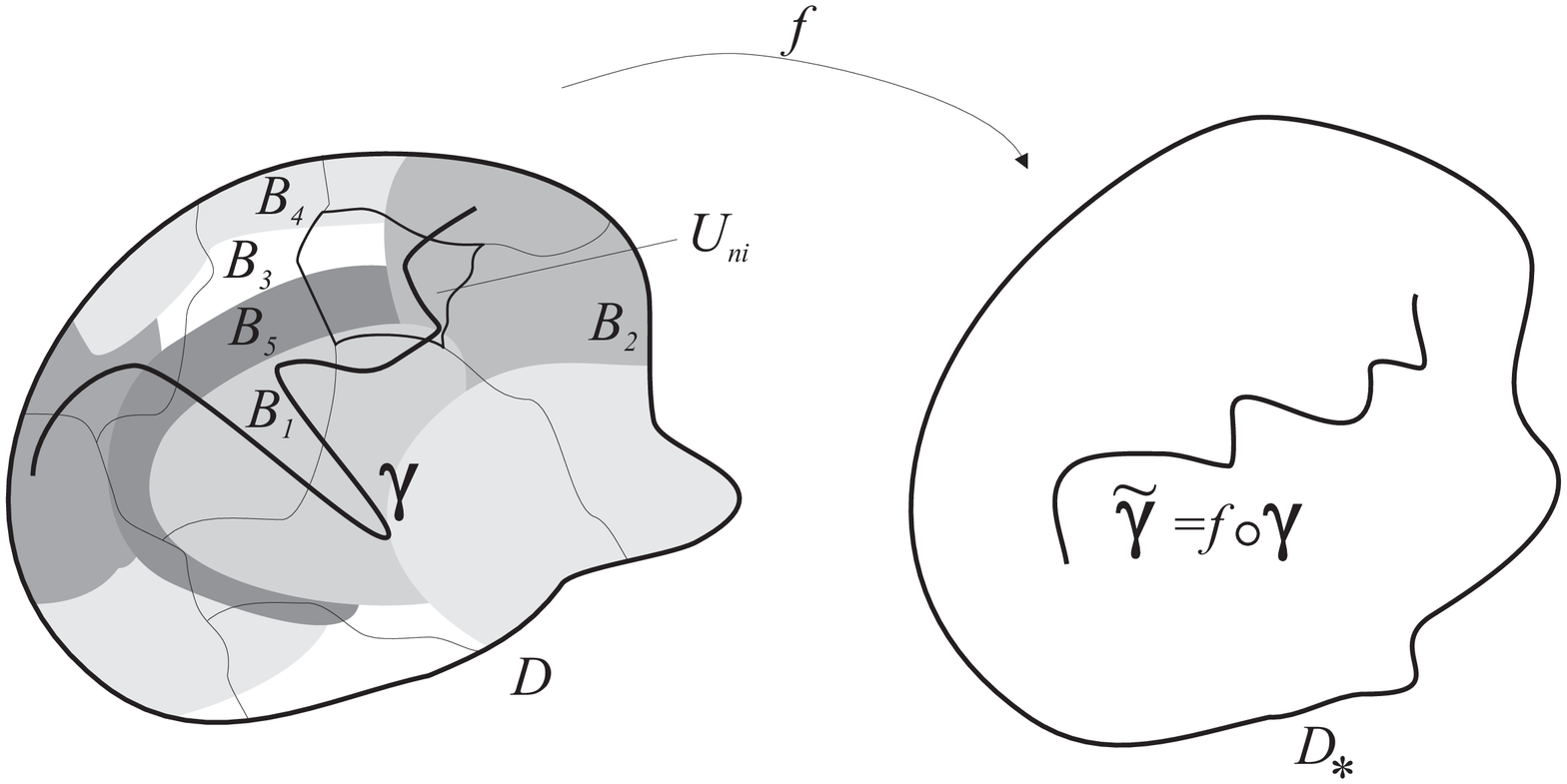}} \caption{К
доказательству теоремы~\ref{th1}}\label{fig1}
\end{figure}
Заметим, что равенство $U_{n_1i_1}=U_{n_2i_2}$ возможно лишь при
$n_1=n_2$ и $i_1=i_2.$ В самом деле, пусть $p\in U_{n_1i_1}\cap
U_{n_2i_2}.$ Тогда, в частности, $p\in U^{\,\prime}_{i_1}\cap
U^{\,\prime}_{i_2},$ что возможно лишь при $i_1=i_2,$ поскольку
$U^{\,\prime}_i\cap U^{\,\prime}_j=\varnothing$ при $i\ne j.$ Далее,
из условия $p\in U_{n_1i_1}\cap U_{n_2i_2}$ вытекает также, что
$f(p)\in V^{\,\prime}_{n_1}\cap V^{\,\prime}_{n_2},$ что также
невозможно, ибо $V^{\,\prime}_i\cap V^{\,\prime}_j=\varnothing$ при
$i\ne j.$ Значит, одновременно $i_1=i_2$ и $n_1=n_2,$ что и
требовалось.

Полагаем
$$f_{n, i}(p):=(\psi_n\circ f\circ\varphi_i^{\,-1})(\varphi_i(p))\,,\quad p\in U_{n, i}\,.$$
Пусть $\rho\in{\rm adm}\,\Gamma $ и
%
%
%
$$\widetilde{\rho}(p_*)=\chi_{f(D\setminus B_0)}\cdot
\sup\limits_{p\in f^{-1}(p_*)\cap D\setminus B_0}
\rho^* (p)\,,$$ 
где
%
$$\rho^{\,*} (p)= \left \{\begin{array}{rr}\rho
(p)/l(f_{n, i}^{\,\prime}(\varphi_i(p))),
&  \text{при }\quad p\in U_{n, i}\setminus B_0, \\
0, & \text{в других случаях.}\end{array} \right.$$
%
%
Заметим, что $\widetilde{\rho}(p_*)=\sup\limits_{k\in{\Bbb N},
1\leqslant i\leqslant I_0\atop 1\leqslant n\leqslant N_0}\rho_{k, i,
n}(p_*)$, где
$$\rho_{k, i, n}(p_*)= \left \{\begin{array}{rr}
\rho^*(f^{-1}_{k, i, n}(p_*)), &  {\rm при }\,\,\,  p_*\in f(B_k\cap U_{n, i}), \\
0, & \text{в других случаях,}\end{array} \right.
$$
где отображение $f_{k, i, n}=f|_{B_k\cap U_{n, i}},$ $k=1,2,\ldots,$
является инъективным. Отсюда следует, что функция $\widetilde{\rho}$
является борелевой, см.~\cite[разд.~2.3.2]{Fe}.

\medskip
Рассмотрим, прежде всего, случай, когда $\widetilde{\gamma}$ --
замкнутая спрямляемая кривая семейства $f(\Gamma).$ Тогда
$\widetilde{\gamma}:[a, b]\rightarrow {\Bbb S}_*$ и
$\widetilde{\gamma}=f\circ \gamma,$ где $\gamma\in\Gamma.$ Пусть
${\widetilde{\gamma}}^0$ -- нормальное представление кривой
$\widetilde{\gamma},$ и пусть $\gamma^*:[0,
l(\widetilde{\gamma})]\rightarrow D$ есть $f$-представление
относительно $\widetilde{\gamma},$ т.е.,
$f(\gamma^*(s))=\widetilde{\gamma}(s)$ при $s\in [0,
l(\widetilde{\gamma})].$
Заметим, что множество
$$S_{n, i}=\{s\in [0, l(\widetilde{\gamma})]:
\gamma^0(s)\in U_{n, i}\}$$
является открытым в ${\Bbb R}$ как прообраз открытого множества
$U_{n_i}$ при непрерывном отображении~$\gamma^*.$ Значит,
$\widetilde{\gamma}|_{S_{n, i}}$ представляет собой не более, чем
счётное число открытых дуг, длина каждой из которых вычисляется в
координатах $(V^{\,\prime}_n, \psi_n)$ при помощи гиперболической
метрики (см. замечания, сделанные во введении).  Обозначим
$\widetilde{\gamma}_{n, i}:=\widetilde{\gamma}|_{S_{n, i}}.$
Согласно сказанному, $\widetilde{\gamma}_{n,
i}=\bigcup\limits_{l=1}^{\infty}\widetilde{\gamma}^l_{n, i},$ где
$\widetilde{\gamma}^l_{n, i}$ -- некоторая открытая дуга. (Поскольку
кривая $\widetilde{\gamma}$ выбиралась замкнутой, то ровно две из
указанных дуг могут оказаться полуоткрытыми, однако, интервалы вида
$[a, c)$ и $(c, b]$ мы интерпретируем как открытые множества по
отношению к отрезку $[a, b ]$). Поскольку $f$ обладает
$N$-свойством, то $\widetilde{h_*}(B_0^{\,**}\cup f(B_0^{\,*}))=0.$
Пусть ${\widetilde{\gamma}}^0$ -- нормальное представление кривой
$\widetilde{\gamma}.$ Тогда по лемме~\ref{lem1} $[0,
l(\widetilde{\gamma})]=\bigcup\limits_{1\leqslant i\leqslant
I_0\atop 1\leqslant n\leqslant N_0}S_{n, i}\cup B_*,$ где $B_*$
имеет линейную меру нуль. В таком случае,
\begin{equation}\label{eq6}
\int\limits_{\widetilde{\gamma}}
\widetilde{\rho}(p_*)\,ds_{\widetilde{h_*}}(p_*)=\sum\limits_{1\leqslant
i\leqslant I_0\atop 1\leqslant n\leqslant N_0}\int\limits_{S_{n, i}}
\widetilde{\rho}(\widetilde{\gamma}^0(s))\,ds
\end{equation}
для почти всех кривых $\widetilde{\gamma}\in f(\Gamma).$ Поскольку
$\widetilde{h_*}(f(B_0))=0,$ то по лемме~\ref{lem1}
$\widetilde{\gamma}^0(s)\not\in f(B_0)$ при почти всех $s\in [0,
l(\widetilde{\gamma})]$ и почти всех кривых $\widetilde{\gamma}\in
f(\Gamma).$ Тогда для почти всех кривых $\widetilde{\gamma}$ и всех
$\gamma$ таких, что $\widetilde{\gamma}=f\circ \gamma$ мы получим,
что
$$
\int\limits_{S_{n, i}} \widetilde{\rho}(\widetilde{\gamma}^0(s))\,ds
= \int\limits_{S_{n, i}}\sup\limits_{p\in
f^{\,-1}(\widetilde{\gamma}^0(s))\cap D\setminus B_0}
\rho^{\,*}(p)\,ds\geqslant
$$
\begin{equation}\label{eq2*}
\geqslant \int\limits_{S_{n, i}}
\frac{\rho(\gamma^{\,*}(s))}{l(f_{n,
i}^{\,\prime}(\varphi_i(\gamma^*(s))))}\,ds\,.
\end{equation}

Поскольку ${\widetilde{\gamma}}$ спрямляема, то и
${\widetilde{\gamma}}^{\,0}$ спрямляема, в частности,
${\widetilde{\gamma}}^{\,0}(s)$ почти всюду дифференцируема (см.
лемму~\ref{lem3}). Покажем, что $\gamma^{\,*}$ абсолютно непрерывна
для почти всех кривых $\widetilde{\gamma}.$ В самом деле,
$\gamma^{\,*}$ спрямляема для почти всех $\widetilde{\gamma},$
поскольку $f\in ACP^{\,-1}.$ Пусть $L_{\gamma, f}^{\,-1}$ --
функция, упомянутая при определении $ACP^{\,-1}$-свойства. Тогда
$$\gamma^{\,*}\circ l_{\widetilde{\gamma}}(t)=\gamma(t)=\gamma^{\,0}\circ
l_{\gamma}(t)=\gamma^{\,0}\circ L_{\gamma,
f}^{\,-1}\left(l_{\widetilde{\gamma}}(t)\right)$$
Обозначая $s:=l_{\widetilde{\gamma}}(t),$ мы получим, что
\begin{equation}\label{eq15}\gamma^{\,*}(s)=\gamma^{\,0}\circ L_{\gamma, f}^{\,-1}(s)\,.
\end{equation}
Тогда $\gamma^{\,*}$ абсолютно непрерывна в локальных координатах,
поскольку по условию $L_{\gamma, f}^{\,-1}(s)$ абсолютно непрерывна
и
$$\widetilde{h}(\gamma^{\,0}(s_1), \gamma^{\,0}(s_2))\le |s_1-s_2|$$
при всех $s_1, s_2\in [0, l(\gamma)].$ Здесь мы также учитываем, что
локально $\widetilde{h}(\gamma^{\,0}(s_1), \gamma^{\,0}(s_2))$
совпадает с $h(\varphi(\gamma^{\,0}(s_1)),
\varphi(\gamma^{\,0}(s_2)))$ в соответствующих локальных координатах
$(U, \varphi),$ кроме того, $|\varphi(\gamma^{\,0}(s_1))-
\varphi(\gamma^{\,0}(s_2))|\leqslant h(\varphi(\gamma^{\,0}(s_1)),
\varphi(\gamma^{\,0}(s_2)))$ по лемме~\ref{lem3}.

Поскольку $\widetilde{\gamma}^0(s)\not\in f(B_0)$ для почти всех
$s\in [0, l(\widetilde{\gamma})]$ и почти всех кривых
$\widetilde{\gamma},$ то $\gamma^{\,*}(s)\not\in B_0$ для почти всех
$s\in [0, l(\widetilde{\gamma})].$ Следовательно, $\left(f_{n,
i}\left(\varphi_i(\gamma^{\,*}(s))\right)\right)^{\,\prime}$ и
$\left(\varphi_i(\gamma^{\,*}(s))\right)^{\,\prime}$ существуют при
почти всех $s\in [0, l(\widetilde{\gamma})]\cap S_{n, i}$ и каждых
$1\leqslant i\leqslant I_0,$ $1\leqslant n\leqslant N_0.$ Напомним,
что $\widetilde{\gamma}_{n,
i}=\bigcup\limits_{l=1}^{\infty}\widetilde{\gamma}^l_{n, i},$ где
$\widetilde{\gamma}^l_{n, i}:=\widetilde{\gamma}|_{\Delta^l_{n,
i}},$ $\Delta^l_{n, i}=(\alpha^l_{n, i}, \beta^l_{n, i}),$ либо
$\Delta^l_{n, i}=[\alpha^l_{n, i}, \beta^l_{n, i}),$ либо
$\Delta^l_{n, i}=(\alpha^l_{n, i}, \beta^l_{n, i}].$ Заметим, что
\begin{equation}\label{eq11}
l_{\widetilde{\gamma}}(s)=\alpha^l_{n, i}+s_h(s)\qquad
\forall\,\,s\in \Delta^l_{n, i}\,, l=1,2,\ldots\,,
\end{equation}
где $s_h(s)$ обозначает гиперболическую длину кривой
$\psi_n(\widetilde{\gamma}_{\Delta^l_{n, i}})$ на отрезке
$[\alpha^l_{n, i}, s].$ Из~(\ref{eq11}) и по лемме~\ref{lem2}
получаем, что при почти всех~$s\in \Delta^l_{n, i}$
\begin{equation}\label{eq12}
\left|\frac{d}{ds}\left(f_{n,
i}\left(\varphi_i(\gamma^{\,*}(s))\right)\right)\right|=\frac{1-|f_{n,
i}\left(\varphi_i(\gamma^{\,*}(s))\right)|^2}{2}\leqslant
\frac{1}{2}\,.
\end{equation}
С другой стороны, по правилу производной сложной функции, для почти
всех~$s\in \Delta^l_{n, i}$
$$\left|\frac{d}{ds}\left(f_{n,
i}\left(\varphi_i(\gamma^{\,*}(s))\right)\right)\right|=$$
\begin{equation}\label{eq2A}
=|f_{n, i}^{\,\prime}(\varphi_i(\gamma^{\,*}(s)))\cdot
(\varphi_i(\gamma^{\,*}(s)))^{\,\prime}| =\left|f_{n,
i}^{\,\prime}(\varphi_i(\gamma^{\,*}(s)))\cdot\frac{(\varphi_i(\gamma^{\,*}(s)))^{\,\prime}}
{|(\varphi_i(\gamma^{\,*}(s)))^{\,\prime}|}\right|\cdot
|(\varphi_i(\gamma^{\,*}(s)))^{\,\prime}|\geqslant
\end{equation}
$$\geqslant l(f_{n,
i}^{\,\prime}(\varphi_i(\gamma^{\,*}(s))))\cdot|(\varphi_i(\gamma^{\,*}(s)))^{\,\prime}|
\,.$$
Объединяя~(\ref{eq12}) и (\ref{eq2A}), получаем, что для почти всех
$s\in S_{n, i}$
\begin{equation}\label{eq3A}
\frac{\rho(\gamma^{\,*}(s))}{l(f_{n,
i}^{\,\prime}(\varphi_i(\gamma^{\,*}(s))))}\geqslant
2\rho(\gamma^{\,*}(s))\cdot
|(\varphi_i(\gamma^{\,*}(s)))^{\,\prime}|\,.
\end{equation}
Пусть $\gamma^0$ -- нормальное представление кривой $\gamma.$ Тогда,
поскольку $f\in ACP^{\,-1},$ то $\gamma^0(s_0)\not\in
f^{\,-1}(B_0^{\,**})\bigcup B_0^{\,*}$ при почти всех $s_0\in [0,
l(\gamma)]$ и почти всех кривых $\widetilde{\gamma}=f\circ \gamma$
(см.~\cite[теорема~2.10.13]{Fe}). Обозначим
$$Q_{n, i}=\{s_0\in [0, l(\gamma)]:
s_0\in U_{n, i}\}\,.$$

Тогда по абсолютной непрерывности кривой $\gamma^{\,*}(s),$ а также
ввиду соотношений~(\ref{eq14}), (\ref{eq7A}) и (\ref{eq15}),
получаем:
$$1\leqslant
\int\limits_{\gamma}\rho(p)\,ds_{\widetilde{h}}(p)=\sum\limits_{1\leqslant
i\leqslant I_0\atop 1\leqslant n\leqslant N_0}\int\limits_{Q_{n,
i}}\rho(\gamma^0(s_0))\,ds_0=
$$
\begin{equation}\label{eq4A}
=\sum\limits_{1\leqslant i\leqslant I_0\atop 1\leqslant n\leqslant
N_0}\int\limits_{S_{n,
i}}\frac{2\rho(\gamma^{\,*}(s))|(\varphi_i(\gamma^{\,*}(s)))^{\,\prime}|}{1-|\varphi_i(\gamma^{\,*}(s))|^2}\,ds\leqslant\frac{2}{1-r_0^2}
\sum\limits_{1\leqslant i\leqslant I_0\atop 1\leqslant n\leqslant
N_0}\int\limits_{S_{n,
i}}\rho(\gamma^{\,*}(s))|(\varphi_i(\gamma^{\,*}(s)))^{\,\prime}|\,ds\,.
\end{equation}
Объединяя~(\ref{eq6}), (\ref{eq2*}), (\ref{eq3A}) и (\ref{eq4A}),
заключаем, что
$\int\limits_{\widetilde{\gamma}}\widetilde{\rho}(p_*)\,ds_{\widetilde{h_*}}(p_*)\geqslant
1$ для почти всех замкнутых кривых $\widetilde{\gamma}\in
f(\Gamma).$ Случай произвольной кривой $\widetilde{\gamma}$ может
быть получен взятием $\sup$ в выражении
$\int\limits_{\widetilde{\gamma}^{\,\prime}}\widetilde{\rho}(p)\,ds_{\widetilde{h_*}}(p_*)\geqslant
1$ по всем замкнутым подкривым $\widetilde{\gamma}^{\,\prime}$
кривой $\widetilde{\gamma}.$ Следовательно,
$\frac{1}{1-r_0^2}\cdot\widetilde{\rho}\in{\rm adm}\,f(\Gamma).$
Значит,
\begin{equation}\label{eq3*}
M\left(f\left(\Gamma\right)\right)\leqslant
\frac{1}{(1-r_0^2)^2}\int\limits_{D_*}
{\widetilde{\rho}}^{\,2}(p_*)\, dh_*(p_*)\,.
\end{equation}

\medskip
Согласно~\cite[теорема~3.2.5, $m=n$]{Fe}, мы получим, что
$$\int\limits_{U_{n, i}\cap
B_k}
K_f(p)\cdot\rho^2(p)\,d\widetilde{h}(p)=$$$$=4\int\limits_{\varphi_i(U_{n,
i}\cap B_k)} \frac{\Vert(\psi_n\circ f\circ
\varphi^{\,-1}_i)^{\,\prime}(x)\Vert^2}{\det\{(\psi_n\circ f\circ
\varphi^{\,-1}_i)^{\,\prime}(x)\}(1-|x|^2)^2}\cdot\rho^2(\varphi^{\,-1}_i(x))\,dm(x)\geqslant$$
$$\geqslant 4\int\limits_{\varphi_i(U_{n,
i}\cap B_k)} \frac{\Vert(\psi_n\circ f\circ
\varphi^{\,-1}_i)^{\,\prime}(x)\Vert^2}{\det\{(\psi_n\circ f\circ
\varphi^{\,-1}_i)^{\,\prime}(x)\}}\cdot\rho^2(\varphi^{\,-1}_i(x))\,dm(x)=$$
\begin{equation} \label{eq7.3.14}
=4\int\limits_{\psi(f((U_{n, i}\cap
B_k)))}\frac{\rho^{\,2}\left((f_k^{\,-1}\circ\psi_n^{\,-1})(y)\right)}
{\left\{l\left(f^{\,\prime}\left((\varphi_i\circ
f^{\,-1}_k\circ\psi^{\,-1}_n)(y)\right)\right)\right\}^2}\,dm(y)\geqslant
\end{equation}
$$\geqslant (1-R_0^2)^2\int\limits_{f(D)}\rho^{2}_{k, i, n}(p_*)\,d\widetilde{h_*}(p_*)\,.$$
Наконец, по теореме Лебега (см.~\cite[теорема~I.12.3]{Sa}),
учитывая~(\ref{eq3*}) и~(\ref{eq7.3.14}), получаем:
$$\int\limits_D K_f(p)\cdot\rho^2(p)\,d\widetilde{h}(p)=
\sum\limits_{{1\leqslant i\leqslant I_0,}\,\,{1\leqslant n\leqslant
N_0}\atop {1\leqslant k< \infty}}\int\limits_{U_{n, i}\cap B_k}
K_f(p)\cdot\rho^2(p)\,d\widetilde{h}(p)\geqslant$$
$$\geqslant (1-R_0^2)^2\int\limits_{f(D)}\sum\limits_{{1\leqslant i\leqslant I_0,}\,\,{1\leqslant n\leqslant
N_0}\atop {1\leqslant k< \infty}}\rho_{k, i,
n}^2(p_*)\,d\widetilde{h_*}(p_*)\geqslant$$
$$\geqslant (1-R_0^2)^2\cdot
\int\limits_{f(D)}\sup\limits_{{1\leqslant i\leqslant
I_0,}\,\,{1\leqslant n\leqslant N_0}\atop {1\leqslant k< \infty}}
\rho_{k, i, n}^2(p_*)\,d\widetilde{h_*}(p_*)=$$
$$=(1-R_0^2)^2\cdot\int\limits_{f(D)}{\widetilde{\rho}}^{\,2}(p_*)\,d\widetilde{h_*}(p_*)
\geqslant (1-R_0^2)^2(1-r_0^2)^2\cdot M(f(\Gamma))\,.$$
Окончательно, соотношение
\begin{equation}\label{eq1B}
M(f(\Gamma))\leqslant c\cdot
\int\limits_{D}K_f(p)\cdot\rho^2(p)\,d\widetilde{h}(p)
\end{equation}
выполняется при для каждой $\rho\in {\rm adm\,}\Gamma,$ где
$c:=\frac{1}{(1-R_0^2)^2(1-r_0^2)^2}.$ Устремляя в~(\ref{eq1B})
$r_0$ и $R_0$ к нулю, получаем желанное соотношение~(\ref{eq1A}).
Теорема~\ref{th1} доказана.~$\Box$

\medskip
{\bf 4. Граничное поведение отображений.} В заключение рассмотрим
приложение теоремы~\ref{th1} к вопросу о граничном поведении
отображений. Пусть $D$ -- область в ${\Bbb S},$  и пусть $E,$
$F\subset D$ -- произвольные множества. В дальнейшем через
$\Gamma(E,F, D)$ мы обозначаем семейство всех кривых
$\gamma:[a,b]\rightarrow D,$ которые соединяют $E$ и $F$ в $D,$ т.е.
$\gamma(a)\in E,\,\gamma(b)\in F$ и $\gamma(t)\in D$ при
$t\in(a,\,b).$ Условимся говорить, что граница $\partial D$ области
$D$ является {\it сильно достижимой в точке $p_0\in
\partial D$}, если для каждой окрестности $U$ точки $p_0$ найдётся компакт
$E\subset D,$ окрестность $V\subset U$ этой же точки и число $\delta
>0$ такие, что для любых континуумов $E$ и $F,$ пересекающих как
$\partial U,$ так и $\partial V,$ выполняется неравенство
$M(\Gamma(E, F, D))\geqslant \delta.$ Мы также будем говорить, что
граница $\partial D$  является {\it сильно достижима}, если она
является сильно достижимой в каждой своей точке. Для множества
$E\subset {\Bbb S},$ как обычно,
$$C(f, E)=\{p_*\in{\Bbb S}_*: \exists\,\, p_k\in D, p\in \partial E:
p_k\rightarrow p, f(p_k)\rightarrow p_*, k\rightarrow\infty\}\,.$$
Следуя~\cite[разд.~2]{IR} (см.~также~\cite[разд.~6.1, гл.~6]{MRSY}),
будем говорить, что функция ${\varphi}:D\rightarrow{\Bbb R}$ имеет
{\it конечное среднее колебание} в точке $p_0\in D$, пишем
$\varphi\in FMO(p_0),$ если
%
%
%
%
$$\limsup\limits_{\varepsilon\rightarrow
0}\frac{1}{\widetilde{h}(\widetilde{B}(p_0,
\varepsilon))}\int\limits_{\widetilde{B}(p_0,\,\varepsilon)}
|{\varphi}(p)-\overline{\varphi}_{\varepsilon}|\
d\widetilde{h}(p)<\infty\,,$$
%
%
где
$\overline{{\varphi}}_{\varepsilon}=\frac{1}
{\widetilde{h}(\widetilde{B}(p_0,
\varepsilon))}\int\limits_{\widetilde{B}(p_0, \varepsilon)}
{\varphi}(p) \,d\widetilde{h}(p).$ Справедливо следующее
утверждение.

\medskip
\begin{theorem}\label{th2}{\sl\, Пусть $D$
и $D_{\,*}$~--- области римановых поверхностей ${\Bbb S}$ и ${\Bbb
S}_*,$ соответственно, при этом, $\overline{D}$ и
$\overline{D_{\,*}}$ являются компактами. Пусть также $f$ --
открытое дискретное дифференцируемое почти всюду отображение области
$D$ на $D_*,$ принадлежащее классу~$ACP^{\,-1}$ и обладающее $N$ и
$N^{\,-1}$-свойствами Лузина. Предположим, область $D$ локально
линейно связна в точке $b\in \partial D,$ $C(f,
\partial D)\subset \partial D^{\,\prime},$ и $\partial D^{\,\prime}$ сильно достижима хотя бы в одной своей точке
$p_*\in C(f, b).$ Если $Q\in FMO(b)$ и, кроме того, $Q$
удовлетворяет (\ref{eq7}) в точке $b,$ то $C(f, b)=\{p_*\}.$
}
\end{theorem}

\medskip
\begin{proof}
По теореме~\ref{th1} отображение $f$ удовлетворяет
соотношению~(\ref{eq1B}) для каждого семейства кривых $\Gamma$ в
области $D.$ В частности, для любых двух континуумов $C_0\subset
\overline{\widetilde{B}(b, r_1)},$ $C_1\subset {\Bbb S}\setminus
\widetilde{B}(b, r_2)$ выполняется условие
\begin{equation}\label{eq15A}
M(f(\Gamma(C_1, C_0, A)))\leqslant\int\limits_{A\cap
D}Q(p)\cdot\rho^2(p)\,d\widetilde{h}(p)\quad\forall\,\rho\in{\rm
adm\,}\Gamma(C_1, C_0, A)\,,
\end{equation}
$$A=A(b, r_1, r_2)=\{p\in {\Bbb S}: r_1<d(p, p_0)<r_2\}, \quad 0 < r_1 <
r_2 <\infty\,.$$
%
\end{proof}
Пусть $\eta:(r_1, r_2)\rightarrow [0, \infty]$ -- произвольная
измеримая по Лебегу функция, удовлетворяющая условию
$\int\limits_{r_1}^{r_2}\eta(t)\,dt\geqslant 1.$ Положим
$\rho(p)=\eta(\widetilde{h}(p, p_0)),$ тогда для произвольной
(локально спрямляемой) кривой $\gamma\in \Gamma(C_1, C_0, A)$
ввиду~\cite[предложение~13.4]{MRSY} выполнено условие
$\int\limits_{\gamma}\rho(p)\,ds_{\widetilde{h}}(p)\geqslant 1.$ В
таком случае,
\begin{equation}\label{eq16A}
M(f(\Gamma(C_1, C_0, A)))\leqslant\int\limits_{A\cap
D}Q(p)\cdot\eta(\widetilde{h}(p, p_0))\,d\widetilde{h}(p)\,.
\end{equation}
Заметим, что каждая кривая $\beta:[a, b)\rightarrow D_*$ имеет
максимальное $f$-поднятие с началом в точке $p\in
f^{\,-1}(\beta(a))$ в области $D$ (см.~\cite[лемма~2.1]{SM}).
Заметим также, что римановы поверхности локально регулярны по
Альфорсу (см., напр., \cite[теорема~7.2.2]{Berd}). В таком случае,
необходимое заключение вытекает из~\cite[теорема~5]{Sev$_1$}.~$\Box$


КОНТАКТНАЯ ИНФОРМАЦИЯ

\medskip
\noindent{{\bf Евгений Александрович Севостьянов} \\
{\bf 1.} Житомирский государственный университет им.\ И.~Франко\\
кафедра математического анализа, ул. Большая Бердичевская, 40 \\
г.~Житомир, Украина, 10 008 \\
{\bf 2.} Институт прикладной математики и механики
НАН Украины, \\
отдел теории функций, ул.~Добровольского, 1 \\
г.~Славянск, Украина, 84 100\\
e-mail: esevostyanov2009@gmail.com}

\end{document}